\documentclass[12pt]{amsart}
\usepackage{amsmath,amssymb,amsthm}

\numberwithin{equation}{section}

\newtheorem{thm}{Theorem}[section]
\newtheorem{lem}{Lemma}[section]

\newcommand{\n}{\nonumber}

\newcommand{\si}{\sigma_R }

\renewcommand{\a}{\alpha}
\renewcommand{\l}{\lambda}

\newcommand{\s}{\sigma}

\renewcommand{\O}{\Omega}
\newcommand{\bb}{\begin{equation}}
\newcommand{\ee}{\end{equation}}
\newcommand{\bq}{\begin{eqnarray}}
\newcommand{\eq}{\end{eqnarray}}
\newcommand{\bqn}{\begin{eqnarray*}}
\newcommand{\eqn}{\end{eqnarray*}}
\newcommand{\R}{\mathbb{R}}
\newcommand{\I}{\infty}
\newcommand{\pd}{\partial}
\renewcommand{\div}{\mathop{\mathrm{div}}}
\newcommand{\curl} {\mathop{\mathrm{curl}}}
\newcommand{\sz}{S_0} 

\begin{document}

\title[Discretely self-similar solutions]{On discretely self-similar solutions of the Euler equations}
\author{Dongho Chae}
\author{Tai-Peng Tsai}
\thanks{This work was initiated when both authors attended a workshop held in
the National Center for Theoretical Sciences, Hsinchu, in December
2012.  We thank R.~Kohn for the reference \cite{EggFon}.  The research
of Chae is supported partially by NRF Grant no. 2006-0093854, while
the research of Tsai is supported in part by a grant from the Natural
Sciences and Engineering Research Council of Canada.}
\address[D. Chae]{Chung-Ang University, Department of Mathematics,
Dongjak-gu Heukseok-ro 84, Seoul 156-756, Republic of Korea}%
\email{dchae@cau.ac.kr}
\address[T.-P. Tsai]{University of British Columbia, Department of Mathematics,
1984 Mathematics Road,
    Vancouver, B.C. V6T 1Z2, Canada} %
\email{ttsai@math.ubc.ca}

\begin{abstract}
This notes gives several criteria which exclude the existence of
discretely self-similar solutions of the three dimensional
incompressible Euler equations.
\end{abstract}
\maketitle
\section{Introduction}

Let $I =(-\I,0)$ or $I=(0,\I)$ be a time interval.  We are concerned
with the Euler equations for the homogeneous incompressible fluid
flows in $\R^3\times I$,
 \[
\mathrm{(E)}
 \left\{ \aligned
 &\frac{\partial v}{\partial t} +(v\cdot \nabla )v =-\nabla p ,
  \\
 &\quad \div v=0 ,
  \endaligned
  \right.
  \]
where $v=(v_1, v_2, v_3 )$, $v_j =v_j (x, t)$, $j=1,2,3$, is the
velocity of the flow, and $p=p(x,t)$ is the scalar pressure.  It is
called \textit{backward} or \textit{forward} depending on whether
$I= (-\infty,0)$ or $I=(0, \infty)$.  Thanks to the time reversal
symmetry of the Euler equations there is a one to one correspondence
between backward and forward solutions, and we may only consider the
backward case $I=(-\infty,0)$.

Recall that the system (E) has the scaling property that if $(v, p)$
is a solution of the system (E), then for any $\lambda >0$ and $\alpha
\in \R $ the functions
\begin{equation}
\label{1.1}
v^{\lambda, \alpha}(x,t)=\lambda ^\alpha v (\lambda x, \l^{\a +1}
t),\quad p^{\l, \a}(x,t)=\l^{2\a}p(\l x, \l^{\a+1} t )
\end{equation}
are also solutions of (E). One can also include space-time translation
in \eqref{1.1}, but we omit it for simplicity.  We say that a solution
$(v,p)$ of (E) is \textit{self-similar} (SS) with respect to the
space-time origin $(0,0)$ if there exists $\alpha\in (-1,\I)$ such
that, for all $\l>0$,
\begin{equation}
\label{1.2a}
v^{\lambda,  \alpha}(x,t) = v(x,t),\quad p^{\l, \a}(x,t)=p(x,t) , \quad
(x,t)\in \R^3\times I.
\end{equation}
It follows that $v(x,t) = \frac 1 {|t|^{a}} V(\frac {x}{|t|^b})$ for
 $V(y)=v(y,\text{sign } t)$ and
\begin{equation}
\label{ab.def}
a= \frac {\a}{\a+1},\quad
b=\frac {1}{\a+1}, \quad \a > -1.
\end{equation}
The condition $\alpha>-1$ ensures that the solution concentrates at
the origin as $t \to 0$.  If a solution satisfies \eqref{1.2a} for one
single $\lambda>1$, we say it is \textit{discretely self-similar} (DSS)
with \textit{factor} $\lambda$. It does not need to satisfy \eqref{1.2a}
for every $\lambda$, and a self-similar solution is considered as a
special case. The possibilities of self-similar singularities in the
Euler equations are studied in \cite{Chae04, Chae07a, Chae07b, Chae10,
  Chae11, ChaKanLee, ChaShv}.  The existence of DSS solutions of (E)
has not been studied, and is the main concern of this note.


Our analysis is based on the self-similar transform.  The self-similar
transform with respect to $(0, 0)$ is the map $(v,p) \mapsto (V,P)$
given by
\begin{equation}
\label{1.4}
 v(x, t)=\frac 1{(-t)^{a}} V(y,s), \quad
 p(x,t)=\frac{1}{b(-t)^{2a}} P(y,s),
\end{equation}
where $a\in \R$ and $b>0$ are given by \eqref{ab.def}, and
\begin{equation}
\label{1.5} y = (-t)^{-b} x, \quad s = - \log (-t).
\end{equation}
Substituting (\ref{1.4})--(\ref{1.5}) into (E), we obtain the
following system for $(V,P)$:
\begin{equation}\label{dse} \left\{
\aligned
&\frac {\partial V}{\partial s}+ \frac {\a}{\a+1} V +\frac {1}{\a+1}(y \cdot \nabla)V + (V\cdot \nabla )V =-\nabla
P,\\
& \mathrm{div}\,V=0.
\endaligned \right.
\end{equation}
A solution $(v,p)$ of (E) is self-similar if and only if $(V,P)$ is
independent of $s$.  A solution $(v,p)$ of (E) is discretely
self-similar with factor $\l>1$ if and only if
\begin{equation}
\label{1.6}
V(y,s)= V(y,s+{\sz}), \quad \forall (y,s)\in \R^{3+1}
\end{equation}
where ${\sz}=(\alpha+1)\log \lambda>0$. In other words, $(V, P)$ is a
time periodic solution of (\ref{dse}) with period ${\sz}$.

Discretely self-similar solutions of partial differential equations
have been considered in many other contexts such as the cosmology.
See the review \cite[Section 5]{EggFon} and the references therein.

We now sketch the structure of the rest of the paper.  In Subsection
\ref{sec:1-1} we review related results for Navier-Stokes equations.
In Section \ref{sec:2} we give nonexistence criteria based on
vorticity integrability, and we will state and prove Theorems
\ref{th:1} and \ref{th:2}.  In Section \ref{sec:3} we give
nonexistence criteria based on velocity integrability, and we will
state and prove Theorems \ref{th:3} and \ref{th:4}.

\subsection{Related results for Navier-Stokes equations}
\label{sec:1-1}
For comparison, we review related results for Navier-Stokes equations
(NS) for which we add $\Delta v$ to the right side of (E). For (NS),
the backward and forward cases are very different.  Introduce the
similarity variables: We take parameter $a<0$ for backward case and
$a>0$ for forward case, and let
\begin{equation}
\label{SS.formula}
v(x,t)= \frac 1{\sqrt{2at}} V(y, s), \quad p(x,t)= \frac 1{{2at}} P
(y, s),
\end{equation}
where
\begin{equation}
y =  \frac x{\sqrt{2at}}, \quad s = \frac 1{2a}\log(2at).
\end{equation}
The corresponding time-dependent \textit{Leray's equations} for $(V,P)$ read
\begin{equation}
\frac {\partial V}{\partial s}-\Delta U -a V - ay \cdot \nabla V +
(V\cdot \nabla )V =-\nabla P, \quad \div V=0.
\end{equation}

For the forward case $a>0$, one can consider the Cauchy problem for
(NS) with initial data $v_0(x)$ which is also SS or DSS, i.e., it
satisfies \eqref{1.2a} with no time dependence. For small data, the
unique existence of small mild solutions by \cite{GM, CMP, CP} implies
those with SS or DSS data are also SS or DSS. For large SS data, the
corresponding SS solution has recently been constructed by
\cite{JiaSve}, and extended to large DSS data by \cite{Tsai12} if the
DSS-factor $\l$ is sufficiently close to $1$ according to the size of
the data.

For the backward case $a<0$, the existence question of self-similar
solutions was raised by Leray \cite{Leray}. It was excluded if $V \in
L^3(\R^3)$ by \cite{NRS}, and if $V \in L^q(\R^3)$, $3<q\le \I$, by
\cite{Tsai98}. Further extensions were given in \cite{Chae06, Chae07a,
  Chae10, ChaKanLee}. The existence of backward DSS solutions has not
been addressed in literature, except that if $V \in L^\infty
(\R,L^3(\R^3))$, which is equivalent to $v\in L^\infty
((-\I,0),L^3(\R^3))$, it must be zero by the result of
\cite{ESS}. Thus one is concerned, e.g., if $V$ only has the bound
\begin{equation}
|V(y,s)| \le \frac {C_*}{1+|y|}, \quad \forall (y,s)\in \R^{3+1},
\end{equation}
for some large constant $C_*$.  A special case that $V(y,s)=R(s\vec
k)\tilde V(y)$, with $R(s\vec k)$ being the rotation about a fixed
axis $\vec k$ by angle $s|\vec k|$, was proposed by G.~Perelman.%
\footnote{private communication of G.~Seregin.} Then $\tilde V$ satisfies a
time-independent system. As pointed out to one of us by R.~Kohn, the
examples of Scheffer \cite{Sch85} are DSS solutions with singular DSS
forces. In view of the forward case, one may hope that the case with
the DSS-factor $\l$ sufficiently close to $1$ might be easier to
exclude.

\section{Criteria based on vorticity}
\label{sec:2}
In this section we give nonexistence criteria based on vorticity
integrability.

\begin{thm}
\label{th:1}
Let $V(y,s)\in C^2_x C_t^1 (\R^{3+1})$ be a time
periodic solution of \eqref{dse} with period ${\sz}>0$ that has
bounded first derivatives and satisfies
\begin{equation}\label{cond1}
\lim_{|y|\to \infty}  V(y,s)=0,\quad \forall s\in [0,\sz),
\end{equation}
and for some $r>0$,
\begin{equation} \label{cond2}
\O:=\curl V\in 
\bigcap_{0<q<r} L^q (\R^3\times [0,\sz]).
\end{equation}
Then $V=0$ on $\R^{3+1}$.
\end{thm}


\begin{proof}
We first observe that from the calculus identity
 \[ V(y,s)=V(0,s)+\int_0 ^1\partial_\tau V(\tau y,s) d\tau=V(0,s)+\int_0 ^1 y\cdot
 \nabla V(\tau y,s) d\tau,
 \]
 we have
 $|V(y,s)|\leq |V(0,s)|+ |y|\|\nabla V(s)\|_{L^\infty}$,
and hence
\begin{equation}\label{cal}
\sup_{(y,s)\in \R^{3+1}} \frac{|V(y,s)|}{1+|y|} \leq C_1:= \max_s
|V(0,s)|+ \|\nabla V\|_{L^\infty(\R^{3+1})}.
\end{equation}

Let us consider the radial cut-off function $\sigma\in C_0
^\infty(\R^N)$ such that
\begin{equation}\label{16}
   \sigma(|x|)=\left\{ \aligned
                  &1 \quad\mbox{if $|x|<1$},\\
                     &0 \quad\mbox{if $|x|>2$},
                      \endaligned \right.
\end{equation}
and $0\leq \sigma (x)\leq 1$ for $1<|x|<2$.  Then, for each $R >0$, we
define $\s_R (y):= \s \left(\frac{|y|}{R}\right)\in C_0 ^\infty
(\R^N).  $ We operate curl on (\ref{dse}),
\begin{equation}\label{vor}
\frac{\pd}{\pd s} \O+\O+ \frac{1}{\alpha+1}( y\cdot \nabla ) \O+(V\cdot \nabla
 )\O=(\O\cdot \nabla )V .
\end{equation}
We multiply (\ref{vor}) by $|\O|^{q-2}\O \si $, and integrate over
$\R^3 \times (0, {\sz})$. The first term vanishes by periodicity.
After integration by parts we get
\begin{align}\label{sel}
&(1 -\frac{3}{(\alpha+1)q}) \int_0 ^{{\sz}} \int |\Omega|^q \si\, dy\,ds
  -\int_0 ^{{\sz}}\int_{\R^3} \xi\cdot \nabla V\cdot \xi |\O(s)|^q
  \si\, dy\,ds\n \\ &=I:= \frac 1q \int_0 ^{{\sz}}\int_{\R^3} |\O(s)|^q (\frac
  y{\a+1}+V)\cdot\nabla \si \, dy\,ds
\end{align}
where we set $\xi=\O/|\O|$.  Since $|\nabla \si(y)|\le \frac
{\|\sigma'\|_{L^\infty}}R 1_{R \le |y| \le 2R}$, by \eqref{cal} we
have
\begin{equation}
|I| \le C (1+C_1) \int_0 ^{{\sz}}\int_{\{R\leq |y|\leq 2R\}} |\O|^q \,
dy\,ds.
\end{equation}
Passing $R\to \infty$, we have $I \to 0$
and (\ref{sel}) gives
\begin{equation}\label{sel2}
\left|1-\frac{3}{(\alpha+1)q}\right|\cdot \int_0
^{\sz}\|\O(s)\|_{L^q}^q ds \le C_1\int_0 ^{\sz}\|\O(s)\|_{L^q}^qds.
\end{equation}
This is true for all $q\in (0, r)$ for some $r>0$. Passing
$q\downarrow 0$ in (\ref{sel2}), we get $\int_0
^{\sz}\|\O(s)\|_{L^q}^qds=0$.  Therefore $\O=\mathrm{curl}\, V=0$ on
$\R^3\times (0, {\sz})$. This, together with $\mathrm{div}\,V=0$,
provides us with the fact that $V (\cdot ,s)=\nabla h(\cdot, s)$ for
all $s\in [0, {\sz})$ for a scalar harmonic function $h(\cdot, s)$ on
  $\R^3$. Using \eqref{cond1} we have $V(\cdot,s)=0$ by Liouville
  theorem for harmonic functions.
\end{proof}

\noindent{\em Remark after the proof: } The above proof works for more
general system
\begin{equation}
\left\{ \aligned &V_s+ a V +b (y \cdot \nabla)V + (V\cdot \nabla )V
=-\nabla P,\\ & \mathrm{div}\,V=0, \endaligned \right.
\end{equation}
where $a,b$ are arbitrary real constants with $b\neq 0$.

\begin{thm}
\label{th:2}
Let $V(y,s)\in C^2_x C_t^1 (\R^{3+1})$ be a time periodic solution of
\eqref{dse} with period ${\sz}>0$ that has bounded first derivatives
satisfying
\begin{equation}\label{cond3}
\lim_{|y|\to \infty} \sup_{0 <s <\sz} |V(y,s)| + | \nabla V(y,s)| =0,
\end{equation}
and  there exists $q\in (0, \frac{3}{1+\a})$ such that
\begin{equation}\label{cond4}
 \O\in L^q (\R^3\times [0,\sz]).
\end{equation}
Then, $V=0$ on $\R^{3+1}$.
\end{thm}

\begin{proof}
Writing (\ref{cond4}) in terms of spherical coordinates,
\[
\int_0^{{\sz}}\int_{\R^3}|\O|^q dy\,ds =\int_{0} ^\infty \int_0
^{{\sz}}\int_{|y|=r} |\O|^q\, dS_r\, ds\,dr < \infty ,
\]
one finds that there exists a sequence $R_j\uparrow \infty$ such that
\begin{equation}
R_j \int_0 ^{{\sz}}\int_{|y|=R_j} |\O |^q dS_{R_j} ds\to 0 \quad
\mbox{as }j\to \infty.
\end{equation}
 We multiply (\ref{vor}) by $\O|\O|^{q-2}$ and rewrite it
\begin{eqnarray}\label{vor1}
 &&\frac{1}{q} \frac{\partial}{\partial s}|\O|^q +|\O|^q+\frac{1}{q(\a +1)}\mathrm{ div }\, (y|\O|^q) - \frac{3}{q(\a +1)}|\O|^q \n \\
 &&= \hat{\alpha} |\O|^{q}-\frac{1}{q} \mathrm{div} \,(V |\O|^q) ,
\end{eqnarray}
where $\hat{\alpha} = \xi \cdot \nabla V \cdot\xi $ with
$\xi=\O/|\O|$. Note that $|\hat{\alpha}|\leq |\nabla V|$.  Let us fix
an $R>0$ and integrate (\ref{vor1}) over the domain $(y,s)\in \{
R<|y|< R_j \}\times(0, {\sz})$. Applying the divergence theorem, we
have
\begin{multline*}
\left(\frac{3}{q(\a+1)} -1\right)\int_0 ^{{\sz}} \int_{R<|y|<R_j} |\O|^q dy\,ds
  +\frac{R}{q(\a+1)}\int_0 ^{{\sz}} \int_{|y|=R} |\O|^q dS_Rds\\
- \frac{R_j}{q(\a+1)} \int_{|y|=R_j} |\O|^q
  dS_{R_j}ds\\
=-\int_0 ^{{\sz}}\int_{R<|y| <R_j} \hat{\alpha} |\O|^q dy\,ds
- \frac{1}{q}\int_0 ^{{\sz}}\int_{|y|=R} V_r|\O|^qdS_R ds +\frac{1}{q}\int_0 ^{{\sz}}\int_{|y|=R_j} V_r |\O|^q dS_{R_j}ds,
\end{multline*}
where $V_r= V\cdot \frac y{|y|}$ and we used the fact $\int_0 ^{{\sz}}
(\|\O(s)\|_{L^q}^q)_s ds=0$, following from the periodicity
hypothesis.  Passing $j\to \infty$, one obtains
 \begin{eqnarray}
 &&\left(\frac{3}{q(\a+1)} -1\right)\int_0 ^{{\sz}} \int_{|y|>R} |\O|^q dy\,ds
  +\frac{R}{q(\a+1)} \int_0 ^{{\sz}}\int_{|y|=R} |\O|^q dS_Rds \nonumber \\
  &&=-\int_0 ^{{\sz}}\int_{|y| >R} \hat{\alpha} |\O|^q dy\,ds- \frac{1}{q}\int_0 ^{{\sz}}\int_{|y|=R} V_r|\O|^qdS_Rds. \label{eq2.14}
\end{eqnarray}
By using \eqref{cond3} and choosing $R$ sufficiently large, we have
\[
|\hat{\alpha}| \le  \frac12 \left( \frac{3}{q(\a+1)} -1\right), \quad
|V_r| \le \frac R{2(\alpha+1)},
\]
on the right side of \eqref{eq2.14}.
Consequently,
 \[
\int_0 ^{{\sz}}\int_{|y|>R} |\O|^q dy\,ds = \int_0 ^{{\sz}}\int_{|y|=R} |\O|^p dS_Rds =0,
\] and hence, $\O=0$ on $\{ y\in \R^3\, |\, |y|>R\}\times (0, {\sz})$.
Thus our vorticity $\O$ satisfies the condition (\ref{cond2}) of
Theorem \ref{th:1}. Applying Theorem \ref{th:1} we obtain $V=0$ on
$\R^{3+1}$.
\end{proof}

\section{Criteria based on velocity}
\label{sec:3}
In this section we give nonexistence criteria based on velocity
integrability.

We will need to estimate the pressure $P(y,s)$, which satisfies
\begin{equation}
\label{P.eq}
-\Delta_y P(\cdot,s) = \sum_{i,j} \pd_i \pd_j (V_iV_j(\cdot,s))
\end{equation}
by taking the divergence of \eqref{dse}.
One solution of \eqref{P.eq} is given by
%
\begin{equation}
\label{P.formula}
\tilde P(y,s) = - \frac {|v(y,s)|^2}3 +\sum_{i,j} p.v.\!\int_{\R^3}
K_{ij}(y-z)V_iV_j(z,s)dz,
\end{equation}
where the kernel is
\[
K_{ij}(y) = \frac {3y_iy_j - \delta_{i,j} |y|^2}{4\pi |y|^5}.
\]
In general, for each fixed $t$, the difference $P -\tilde P$ is a
harmonic function in $x$ and may not be constant.
We will assume $P=\tilde P$.

\begin{thm}
\label{th:3}
Suppose that $(V,P)\in C^1_{loc} (\R^{3+1})$ is a time periodic
solution of (\ref{dse}) with period ${\sz}>0$, that the pressure $P$
is given by \eqref{P.formula}, and that $V$ satisfies for some $3\le r
\le 9/2$ one of the the following conditions,
\begin{itemize}
\item[(i)] $\alpha >3/2$ and $V\in L^{3}(0, {\sz}; L^{r} (\R^3))$, or
\item[(ii)] $-1<\alpha<3/2$ and $V\in L^2 (0, {\sz}; L^2 (\R^3))\cap L^{3}(0, {\sz}; L^{r} (\R^3))$.
    \end{itemize}
Then $V=0$ on $\R^{3+1}$.
\end{thm}

\begin{proof}
Since $P$ is given by \eqref{P.formula}, by the Calderon-Zygmund
inequality we have
\begin{equation}\label{cz}
\|P(s)\|_{L^q}\leq C_q\|V(s)\|_{L^{2q}}^2\quad \forall q\in (1, \infty),
\forall s\in \R.
\end{equation}
{\em\underline{ The case (i)}: }  Let $\si$ is the cut-off function introduced in the proof of Theorem 1.1.  We multiply (\ref{dse}) by $V\si$, and integrate over $\R^3\times (0, {\sz})$, then from the time periodicity condition and by integration by part we obtain
\begin{eqnarray}\label{th31}
&&\frac{1}{\alpha+1}\left(\alpha-\frac32\right) \int_0 ^{{\sz}} \int_{\R^3} |V|^2 \si dy\,ds
-\frac{1}{2(\alpha+1)}\int_0 ^{{\sz}} \int_{\R^3}|V|^2 y\cdot \nabla\si dy\,ds \n \\
&&\qquad=\frac12  \int_0 ^{{\sz}} \int_{\R^3} |V|^2 V\cdot \nabla \si dy\,ds + \int_0 ^{{\sz}} \int_{\R^3}P V\cdot \nabla \si dy\,ds
\end{eqnarray}
Since $y\cdot \nabla \si \leq 0$ for all $y\in \R^3$, and the first term of the left hand side of (\ref{th31}) is monotone non-decreasing function of $R$, we find that
\begin{eqnarray}
&&\frac{1}{\alpha+1}\left(\alpha-\frac32\right) \int_0 ^{{\sz}} \int_{\R^3} |V|^2 \s_{R_1} dy\,ds \n \\
&&\leq \frac12  \int_0 ^{{\sz}} \int_{\R^3} |V|^2 V\cdot \nabla \s_{R_2} dy\,ds + \int_0 ^{{\sz}} \int_{\R^3}P V\cdot \nabla \s_{R_2} dy\,ds\n \\
&&:=I_1+I_2
\end{eqnarray}
for all $0<R_1<R_2<\infty$. Passing $R_2 \to \infty$, one has
\begin{eqnarray*}
I_1 &\leq & \frac{\|\nabla \s\|_{L^\infty}}{2R_2}
\int_0 ^{{\sz}} \int_{R_2<|y|<2R_2} |V|^3 dy\,ds\n \\
&\leq & \frac{\|\nabla \s\|_{L^\infty}}{2R_2}
\int_{0}^{{\sz}}\left(\int_{R_2<|y|<2R_2} |V|^{r} dy\right)^{\frac 3r} \left(\int_{R_2<|y|<2R_2}  \, dy\right)^{1-\frac3{r}} ds \n \\
&\leq& C R_2^{2-\frac 9r}\int_0 ^{{\sz}} \|V(s)\|_{L^{r} ( R_2 <|y|< 2R_2)} ^3 ds \to 0,
\end{eqnarray*}
and
\begin{eqnarray*}
I_2&\leq& \frac{\|\nabla \s\|_{L^\infty}}{2R_2} \int_0 ^{{\sz}}\int_{R_2<|y|<2R_2} |V| |P|  dy\,ds\n \\
&\leq & \frac{\|\nabla \s\|_{L^\infty}}{2R_2} \int_0 ^{{\sz}}\left(\int_{R_2<|y|<2R_2}|V|^{r} dy\right)^{\frac 1r}\left(\int_{\R^3}|P|^{\frac r2} dy\right)^{\frac 2r}\left(\int_{R_2<|y|<2R_2}  \, dy\right)^{1-\frac3r} ds\n \\
&\leq& C  R_2^{2-\frac 9r} \int_0 ^{{\sz}} \|V\|_{L^{r} (R_2 <|y|<R_2)}\|V\|_{L^{r}}^2 ds\n \\
&\leq& C  R_2^{2-\frac 9r} \left( \int_0 ^{{\sz}} \|V\|_{L^{r} (R_2 <|y|<R_2)}^3 ds\right)^{\frac13} \left(\int_0 ^{{\sz}} \|V\|_{L^{r}}^{3} ds\right)^{\frac23}\to 0,
\end{eqnarray*}
where we used (\ref{cz}). Therefore, we have
\[ \left(\alpha-\frac32\right) \int_0 ^{{\sz}} \int_{\R^3} |V|^2 \s_{R_1} dy\,ds=0 \]
for all $R_1>0$.  This shows that $V=0$ on $\R^3\times (0, {\sz})$.\\

\noindent{\em \underline{The case (ii)}:}  In this case from (\ref{th31}) we have
\begin{eqnarray}\label{th32}
&&\frac{1}{\alpha+1}\left|\alpha-\frac32\right| \int_0 ^{{\sz}} \int_{\R^3} |V|^2 \si dy\,ds
 \leq \frac{1}{2(\alpha+1)}\int_0 ^{{\sz}} \int_{\R^3}|V|^2 | y\cdot \nabla\si|dy\,ds\n \\
&&\qquad+\frac12  \int_0 ^{{\sz}} \int_{\R^3} |V|^2 |V\cdot \nabla \si| dy\,ds + \int_0 ^{{\sz}} \int_{\R^3}|P| | V\cdot \nabla \si| dy\,ds\n \\
&&\qquad :=J_1+J_2+J_3.
\end{eqnarray}
From the above computations we know that $ |J_2|+|J_3|\to 0$ as $R\to \infty$.
For $J_1$ we estimate easily
\[
|J_1|\leq \frac{C}{R}\int_0 ^{{\sz}} \int_{R<|y|<2R}   |V|^2 |y|dy\,ds\leq C \int_0 ^{{\sz}} \|V(s)\|_{L^2 (R<|y|<2R)}^2 ds \to 0
\]
as $R\to \infty$. Hence $\int_0 ^{{\sz}} \int_{\R^3} |V|^2  dy\,ds=0$, and $V=0$ on $\R^{3+1}$.
\end{proof}

The next result, Theorem \ref{th:4}, is an extension of Chae-Shvydkoy
\cite[Theorem 3.2]{ChaShv} to the case of discretely self-similar
solutions.
An important role is played by the following lemma, which extends the
local energy inequality in \cite{ChaShv}.

\begin{lem}
\label{lemma1}
Suppose $(V,P)\in C^1_{loc}(\R^{3+1})$ is a time periodic solution of
\eqref{dse} with period $\sz>0$. Let
$\lambda=e^{b\sz}>1$. For $-\I<s_1<s_2 < \I$, let $l_j = e^{bs_j}$ and
\begin{equation}
\label{I12.def}
I_j= l_j^{2\alpha-3} 
\int_0^{\sz} \int_{\R^3} |V(y,s_j+\tau)| ^2
\sigma(e^{-b(s_j+\tau)} y)dy\,d\tau, \quad (j=1,2).
\end{equation}
We have
\begin{equation}
\label{I12.est}
l_j^{2\alpha-3}\int_0^{\sz} \int_{|y|\le \frac 12 l_j} |V(y,\tau)| ^2 dy\,d\tau\le
I_j \le l_j^{2\alpha-3} \int_0^{\sz} \int_{|y|\le \lambda l_j}
|V(y,\tau)| ^2 dy\,d\tau
\end{equation}
for $j=1,2$, and for some constant $C=C(\sz)$
\begin{equation}
\label{I3.est}
|I_1-I_2| \le C \int_0^{\sz} \int_{\frac12 l_1 \le |y| \le \lambda
  l_2} |y|^{2\alpha-4}(|V|^3+|PV|)(y,s)\, dy\,ds.
\end{equation}
\end{lem}

Note that $\lambda$ is the factor for discrete self-similarity, see
\eqref{1.6}.

\begin{proof}
Let $\sigma(x)$ be a radial function with $\sigma\ge 0$, $\sigma(r)=1$
for $r<1/2$ small and $\sigma(r)=0$ for $r\ge 1$.  Let $t_j =
-e^{-bs_j}$, $j=1,2$.  Testing the Euler equation with $\sigma v$ in
${\R}^3 \times (t_1,t_2)$ we get
\begin{equation}
\int |v(x,t_2)| ^2 \sigma(x)dx  - \int |v(x,t_1)| ^2 \sigma(x)dx
= \int_{t_1}^{t_2} \int (|v|^2+2p)v \cdot \nabla \sigma(x)dx\,dt.
\end{equation}
In self-similar variables \eqref{1.4}--\eqref{1.5} it becomes
\begin{align}
\notag
&e^{(2a-3b)s_2}\int |V(y,s_2)| ^2 \sigma(e^{-bs_2} y)dy - e^{(2a-3b)s_1}\int |V(y,s_1)| ^2 \sigma(e^{-bs_1} y)dy \\
&\quad = \int_{s_1}^{s_2} e^{(3a-3b-1)s} \int(|V|^2+2P)V \cdot
\nabla \sigma(e^{-bs}y)dy\,ds.
  \end{align}

Assume now that $v$ is DSS, so that $V(y,s)$ is periodic in $s$ with
period ${\sz}>0$.

Replacing $s_j$ by $s_j+\tau$,  dividing by $e^{(2a-3b)\tau}$,
and integrating over $\tau\in [0,{\sz}]$, we get
\begin{equation}
\label{I123}
I_1 - I_2 = I_3
\end{equation}
where $I_1$ and $I_2$ are given in \eqref{I12.def}, and
\begin{equation}
I_3 = \int_0^{\sz} e^{-(2a-3b)\tau}\int_{s_1+\tau}^{s_2+\tau}
e^{(3a-3b-1)s} \int(|V|^2+2P)V\cdot \nabla \sigma(e^{-bs}y)dy\,ds\,d\tau.
\end{equation}

The estimate \eqref{I12.est} for $I_1$ and $I_2$ is because that
$\sigma (e^{-b(s_j+\tau)}y)$ is supported in $\frac 12 l_j \le |y| \le
\lambda l_j$, and also using the periodicity.

For $I_3$, since $e^{-(2a-3b)\tau}\le C$ and $\sigma(e^{-bs}y)$ is
supported in $\frac 12 e^{bs}\le |y| \le e^{bs}$,
\begin{equation}
|I_3| \le C  \int_{E} \int_0^{\sz} \int_{s_1+\tau}^{s_2+\tau}
e^{(3a-3b-1)s} Q(y,s) |\nabla \sigma|(e^{-bs}y)ds\,d\tau dy.
\end{equation}
where $E$ denotes the spatial region $E=\{y:\frac 12 l_1\le |y|\le
\lambda l_2\}$ and $Q=|V|^3+|P||V|$.  If we denote by $f(y,s)$ the
integrand, the inner integral
\begin{equation}
\int_0^{\sz} \int_{s_1+\tau}^{s_2+\tau} f(y,s) ds\,d\tau \le
{\sz} \int_{s_1}^{s_2+{\sz}} f(y,s) ds \le
{\sz} \int_{J_y} f(y,s) ds
\end{equation}
where we have used $\sigma(e^{-bs}y)$ is supported in the time interval
\begin{equation}
J_y = \{s: |y|\le e^{bs} \le 2|y|\} = \{s: \frac {\ln |y|}b \le s \le
\frac {\ln |y|}b + \frac {\ln 2}b \}.
\end{equation}
In $J_y$ we have $e^{(3a-3b-1)s} \le C |y|^{\frac {3a-3b-1}b}=C |y|^{2\alpha-4}$. Thus
\begin{equation}
|I_3| \le C \int_E \int_{J_y} |y|^{2\alpha-4}Q(y,s) ds\,dy.
\end{equation}
Let $k$ be the positive integer so that $(k-1){\sz} <\frac {\ln 2}b \le
k{\sz}$. Using the periodicity of $Q(y,s)$ in $s$,
\begin{equation}
|I_3| \le Ck \int_E \int_0^{\sz} |y|^{2\alpha-4}Q(y,s) ds\,dy.
\end{equation}
This shows \eqref{I3.est}.
\end{proof}

\begin{thm}
\label{th:4}
Suppose $(V,P)\in C^1_{loc}(\R^{3+1})$ is a time periodic solution of
\eqref{dse} in $\R^{3+1}$ with period $\sz>0$, $V \in
L^p(\R^3 \times ([0,\sz])$ for some $3 \le p \le \I$, and $P$ is given
by \eqref{P.formula}. If $-1< \alpha \le 3/p$ or $3/2< \alpha < \I$,
then $V = 0$ on $\R^{3+1}$.
\end{thm}

\begin{proof}
The proof of \cite[Theorem 3.2]{ChaShv} goes through with the help of
Lemma \ref{lemma1}.  One adds the temporal integral $\int_0^{\sz}ds$
in front of every spatial integral in its proof.
\end{proof}

\bibliographystyle{habbrv}
\bibliography{euler-biblio}

\end{document}